\documentclass[10pt]{amsart}

  1
\usepackage{eurosym}
\usepackage{amscd,amsfonts,amsmath,amssymb,amsthm,latexsym}
\usepackage{enumerate,array,srcltx,amsmath}
\usepackage{leqno}
\usepackage[dvips]{graphicx}
\usepackage{multicol}
\usepackage{a4}
\usepackage{fancyhdr}
\usepackage{caption}
\usepackage{subcaption}
\usepackage{mathrsfs}
\usepackage{color}
\usepackage{relsize}
\newtheorem{theorem}{Theorem}[section]
\newtheorem{lemma}[theorem]{Lemma}\newtheorem{corollary}[theorem]{Corollary}

\theoremstyle{definition}

\theoremstyle{remark}
\newtheorem{remark}[theorem]{Remark}

\numberwithin{equation}{section}

\newcommand{\dint}{\displaystyle\int}

\makeatother

\begin{document}
\title[Nonlinear Schr\"odinger equation on a star-shaped network]{Exponential stability for the nonlinear Schr\"odinger equation on a star-shaped network}

\author{Ka\"{\i}s Ammari}
\address{UR Analysis and Control of PDEs, UR13ES64, Department of Mathematics, Faculty of Sciences of Monastir,
University of Monastir, 5019 Monastir, Tunisia}
\email{kais.ammari@fsm.rnu.tn}

\author{Ahmed Bchatnia}
\address{UR Analyse non-lin\'eaire et g\'eom\'etrie, UR13ES32, Department of
Mathematics, Faculty of Sciences of Tunis, University of Tunis El-Manar,
2092 El Manar II, Tunisia}
\email{ahmed.bchatnia@fst.utm.tn}

\author{Naima Mehenaoui}
\address{Department of
Mathematics, Faculty of Sciences of Tunis, University of Tunis El-Manar,
2092 El Manar II, Tunisia}
\email{naima.mehenaoui@fst.utm.tn }

\begin{abstract}
In this paper, we prove the exponential stability of the solution of the nonlinear dissipative Schr\"odinger equation on a star-shaped
network and where the damping is localized on one branch and at the infinity.
\end{abstract}

\date{}
\maketitle
\tableofcontents


\section{Introduction}
Dispersive models have long been a question of great interest in a wide range of researchers. One of the most significant current discussions in these models is the nonlinear Schr\"{o}dinger equation. This equation has been studied extensively since the early years of this century. Most of these studies have mainly concentrated on the well-posedness questions, see, for instance \cite{Kato} and stabilization of the energy. The author in \cite{Nat1,Nat2} established an exponential decay rate of the energy in $L^{2}$-level for the nonlinear Schr\"{o}dinger equation with localized damping.

\medskip

In this paper, we derive analogous exponential decay rate in $L^{2}$-level for the nonlinear dissipative Schr\"{o}dinger equation on a star-shaped network, as in Figure \ref{fig}, and where the damping is localized on one branch and at the infinity. More precisely we consider the following initial and boundary-value problem:
\begin{equation}\label{sys}
\left\{
\begin{array}{lll}
i\, \partial_t u_{1}  + \partial^2_x u_{1} + \lambda u_1|u_1|^{\alpha -1} + i a(x) u_1 = 0 \;
\mbox{for} \; x > 0, \, t > 0,\\

 i\, \partial_t u_{i} + \partial^2_x u_{i} + \lambda u_i |u_i|^{\alpha -1} = 0 \;
\mbox{for} \; x > 0, \, t > 0, \; 2\leq i \leq N,\\

\mathlarger{‎‎\sum}_{i=1}^N \partial_x u_i (t,0)=0,  \, t>0,\\

u_i(t,0)=u_j(t,0), \,  \forall \, t>0, \, 1 \leq i,j \leq N, \\

u_i(0,x)=\varphi_i(x),
 \;\; x>0,\;\; 1\leq i \leq N, \\
\end{array}\right.
\end{equation}
where $N \in \mathbb{N}^*, N \geq 3,\; \lambda\in\mathbb{R}^*$ and $\alpha\in \{3, 5\}$ will be treated in this paper. The presence of the damping term in (\ref{sys}) is responsible for the localized mechanism of dissipation of the system since the function $a=a(x)$ is assumed to be in $L^{\infty}(\mathbb{R}_+)$, almost everywhere non-negative function, and to satisfy for some $R>0$, and $\alpha_{0}>0$,
\begin{equation}
\label{dampingh}
  a(x) \geq \alpha_0 >0 \quad \text{for}\; \vert x\vert >R.
\end{equation}
The considered graph consists of a finite number of edges of infinite length attached to a common vertex, each of them being identified with a copy of the positive semi-axis.
In the context of nonlinear Schr\"{o}dinger equation on a star-shaped network, Ali Mehmeti, Ammari and Nicaise in \cite{amm1bis} have proved $L^{\infty}$-time decay estimates. Banica and Ignat in \cite{Banica} proved the same results in the case of trees with the last generation of edges of infinite strips,
with Kirchhoff coupling condition at the vertices. With the same conditions, dispersive estimates were obtained in the case of the tadpole graph in \cite{amm1}.
  The motivation for studying nonlinear propagation in ramified structures comes from several branches of pure and applied science, modeling phenomena such as nonlinear electromagnetic pulse propagation in optical fibers, the hydrodynamic flow, electrical signal propagation in the nervous system, etc.

\medskip

Before a precise statement of our main results, let us introduce some definitions and notations about 1-d networks which will be used throughout the rest of the paper.

\medskip

Let $\mathcal{R}_i$, $i = 1,2,...,N$ be $N$ disjoint sets identified with to $(0, +\infty)$. We set $\mathcal{R}:= \mathlarger{\overset{N}{\underset{i=1}{\cup}}} \overline{\mathcal{R}_k}$. We denote by $f = (f_k)_{k=1,2,...,N} = (f_1, f_2, ..., f_N)$ the functions on $\mathcal{R}$ taking their values in $\mathbb{C}$ and let $f_k$ be the restriction of $f$ to $\mathcal{R}_k$.\\
Define the Hilbert space $\mathcal{H} =\mathlarger{\overset{N}{\underset{i=1}{\oplus}}}L^2(\mathcal{R}_k) = L^2(\mathcal{R})$ with inner product
$$((u_k), (v_k))_{\mathcal{H}} = \mathlarger{‎‎\sum}_{i=1}^N(u_k, v_k)_{L^2(\mathcal{R}_k)}.$$

\begin{center} \label{fig}
\includegraphics[scale=1.40]{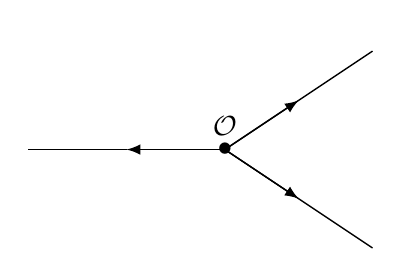}
\captionof{figure}{Star-Shaped Network for $N =3$}
\end{center}

The energy identity obtained from $(\ref{sys})$, by simple formal calculations, is given by
\begin{equation}\label{E}
E_u (t) := \mathlarger{‎‎\sum}_{i=1}^N E_{u_i} (t) = - 2 \, \int_0^t\int_{\mathbb{R}_+} a(x)|u_1(s,x)|^2 \, dxds
+ \mathlarger{‎‎\sum}_{i=1}^N \| \varphi_i \|^2_{L^2(\mathbb{R_+})}
\end{equation}
$$
= - 2 \, \int_0^t\int_{\mathbb{R}_+} a(x)|u_1(s,x)|^2 \, dxds + E_u (0), \, \forall \, t \geq 0,
$$
where $u = (u_1,...,u_N)$ and $E_{u_i}(t)$ is the energy of $u_i$, $1\leq i \leq N,$ defined by $\displaystyle{ E_{u_i}(t):=\int_0^{+\infty}|u_i(x, t)|^2 \, dx.}$\\
Our major concern will be to prove the exponential decay of the global energy at the infinity. More precisely we have the following theorem:
\begin{theorem}\label{theodec}
Consider $\alpha\in \{3,5\}$, a function $a$ satisfying assumption (\ref{dampingh}) and the initial data $\varphi$ in $L^2(\mathcal{R})$. For any solution $u$ of the system (\ref{sys}), there exist $c>0$ and $\omega>0$ such that:
\begin{equation}\label{decay}
E_u (t)\leq c \, e^{-\omega t}, \;\; \mbox{ for all } \; t\geq 0,
\end{equation}
provided the initial data satisfies $\left\|\varphi\right\|_{L^2(\mathcal{R})}\ll 1,$ if one considers the case $\alpha=5.$

\end{theorem}

The paper is organized as follows. In Section 2, we prove that the system (\ref{sys}) is globally well-posed in the energy space  $L^2(\mathcal{R})$. Section 3 is devoted to the proof of the main result. Technical results are collected in the appendix.
\section{Well-posedness}
In this section, we show the global well-posedness in $L^{2}$ of the problem $(\ref{sys})$ for initial data in $L^{2}(\mathcal{R})$, by combining the techniques due to Kato, established in \cite{Cazenav} and \cite{linares}.
We first recall the following Strichartz estimates, see \cite{Cazenav,linares} for more details.
\begin{theorem}\label{strichartz}
The group $\left(e^{it\partial_{x}^{2}}\right)_{t \in \mathbb{R}}$ associated to the Schr\"{o}dinger equation satisfies the following properties:
\begin{enumerate}
\item[1)] $\left(\int_{-\infty}^{+\infty}\left\Vert e^{it\partial_{x}^{2}}f \right\Vert_{L^{p}}^{q}dt \right)^{\frac{1}{q}}\leq c \Vert f \Vert_{L^{2}}, $
\item[2)] $\left(\int_{-\infty}^{+\infty}\left\Vert\int_{-\infty}^{+\infty} e^{i(t-s)\partial_{x}^{2}}h(.,s)ds \right\Vert_{L^{p}}^{q}dt \right)^{\frac{1}{q}}\leq c \left(\int_{-\infty}^{+\infty}\left\Vert h(.,s) \right\Vert_{L^{p'}}^{q'}dt \right)^{\frac{1}{q'}},$
\item[3)] $\left\Vert\int_{-\infty}^{+\infty} e^{it\partial_{x}^{2}}h(.,t)dt \right\Vert_{L^{2}}\leq c \left(\int_{-\infty}^{+\infty}\left\Vert h(.,s) \right\Vert_{L^{p'}}^{q'}dt \right)^{\frac{1}{q'}}.$
\end{enumerate}
In all cases, we have $2\leq p\leq+\infty$ and
\begin{equation}\label{couple}
\frac{2}{q}=\frac{1}{2}-\frac{1}{p}.
\end{equation}
Here, constant $c>0$ depends only on $p$.
\end{theorem}
As a consequence we have the following (see \cite{Cazenav} and \cite{linares}).
\begin{corollary}\label{str}
Consider $(p,q)$  and $(p_{0},q_{0})$ two pairs of constants satisfying condition $(\ref{couple})$. One has
\begin{equation}
 \left(\int_{-T}^{T}\left\Vert\int_{-T}^{T} e^{i(t-s)\partial_{x}^{2}}h(.,s)ds \right\Vert_{L^{p_{0}}}^{q_{0}}dt \right)^{\frac{1}{q_{0}}}\leq c \left(\int_{-T}^{T}\left\Vert h(.,s) \right\Vert_{L^{p'}}^{q'}dt \right)^{\frac{1}{q'}},
\end{equation}
where $c=c(p,p_{0})>0.$
\end{corollary}

We have under the assumption (\ref{dampingh}) the following local well-posedness result in $L^{2}$-level for the nonlinear Schr\"{o}dinger $(\ref{sys})$:
\begin{theorem}\label{Th1}
Given $\varphi = (\varphi_1,...,\varphi_N) \in L^2(\mathcal{R})$ and $\alpha \in (1,5]$, then there exist $T=T(\varphi, \|a\|_{L^{\infty}}, \alpha, \lambda)>0$ and a unique solution $u=(u_1,...,u_N)$ of system (\ref{sys}) such that:
\begin{equation}
u_i\in \mathcal{C}([0,T];L^2(\mathcal{R}_i))\cap L^r\big((0,T);L^{\alpha+1}(\mathcal{R}_i)\big), 1\leq i \leq N,
\end{equation}
where $r=\frac{4(\alpha+1)}{\alpha-1}.$
\end{theorem}
\begin{proof}
We divide the proof in two cases.

\medskip

\item [I)\;\textbf{Subcritical case, $\alpha\in(1,5)$: }] Consider $T$ and $b$ positive constants, we need to construct the complete metric space,
\begin{equation*}
X_{T,b}=\left\lbrace
v\in C\left( \left[0,T \right];L^{2}(\mathcal{R}) \right) \cap  L^r\big((0,T);L^{\alpha+1}(\mathcal{R})\big),\vert \vert\vert v \vert\vert\vert\leq b
\right\rbrace,
\end{equation*}
where $\vert \vert\vert . \vert\vert\vert$ indicates the natural norm of the space $$ C\left( \left[0,T \right];L^{2}(\mathcal{R}) \right) \cap \, L^r\left( (0,T);L^{\alpha+1}(\mathcal{R})\right)$$ given by:
$$\vert \vert\vert v \vert\vert\vert_{T}=\sup_{t\in[0,T]} \Vert v(t)\Vert_{L^{2}} + \left( \int_{0}^{T}\|v(t)\|_{L^{\alpha+1}}^{r}dt\right)^{{1}/{r}},
$$
for $r=\frac{4(\alpha+1)}{\alpha-1}$.

\medskip

\item Step 1. Define, for any $u\in X_{T,b}$
\begin{equation}\label{A1}
\left\{
\begin{array}{lll}
\begin{split}
 \Phi_{\varphi_{1}}(u_{1})(t)=\Phi(u_{1})(t)=e^{it\partial_{x}^{2}}\varphi_{1}+i\int_{0}^{T} e^{i\partial_{x}^{2}(t-s)}\left( \lambda \vert u_{1}\vert^{\alpha-1}u_{1}+i\,a(.)u_{1}\right) (s)ds, \\
\Phi_{\varphi_{i}}(u_{i})(t)=\Phi(u_{i})(t)=e^{it\partial_{x}^{2}}\varphi_{i}+i\int_{0}^{T} e^{i\partial_{x}^{2}(t-s)}\left( \lambda \vert u_{i}\vert^{\alpha-1}u_{i}\right) (s)ds\;, \; 2\leq i \leq N. \\
\end{split}
\end{array}\right.
\end{equation}
By using Theorem $\ref{strichartz}$, we deduce from the definition $\Phi(.)$ in equation $\ref{A1}$ and Corollary $\ref{str}$ that
\begin{equation}\label{A3}
\left\{\begin{array}{lll}
 \left\Vert\Phi(u_{1})(t)\right\Vert_{L^{r}\left(L^{\alpha+1}\right)}
\leq \\
c\left\Vert\varphi_{1}\right\Vert_{L^{2}}+c \Vert a \Vert_{L^{\infty}}T \underset{t\in[0,T]}{\sup} \left\Vert u_{1}(t)\right\Vert_{L^{2}} + c\left\vert\lambda\right\vert \left(\int_{0}^{T} \left\Vert u_{1}(t)\right\Vert_{L^{\alpha+1}}^{\alpha r'} dt\right)^{{1}/{r'}},\\
\left\Vert\Phi(u_{i})(t)\right\Vert_{L^{r}\left(L^{\alpha+1}\right)}
\leq c\left\Vert\varphi_{i}\right\Vert_{L^{2}} + c\left\vert\lambda\right\vert \left(\int_{0}^{T} \left\Vert u_{i}(t)\right\Vert_{L^{\alpha+1}}^{\alpha r'} dt\right)^{{1}/{r'}}\;, \; 2\leq i \leq N, \\
\end{array}\right.
\end{equation}
where $r'=\frac{r}{r-1}$.
H\"{o}lder's inequality gives
\begin{equation}\label{A4}
\left\{\begin{array}{lll}
 \left\Vert\Phi(u_{1})(t)\right\Vert_{L^{r}\left(L^{\alpha+1}\right)}
\\
\leq C_{1}\left( \left\Vert\varphi_{1}\right\Vert_{L^{2}}+ T \underset{t\in[0,T]}{\sup} \left\Vert u_{1}(t)\right\Vert_{L^{2}} + T^{\theta} \left(\int_{0}^{T} \left\Vert u_{1}(t)\right\Vert_{L^{\alpha+1}}^{r} dt\right)^{{\alpha}/{r}}\right), \\
\left\Vert\Phi(u_{i})(t)\right\Vert_{L^{r}\left(L^{\alpha+1}\right)}
\leq C_{2}\left( \left\Vert\varphi_{i}\right\Vert_{L^{2}}+ T^{\theta} \left(\int_{0}^{T} \left\Vert u_{i}(t)\right\Vert_{L^{\alpha+1}}^{r} dt\right)^{{\alpha}/{r}}\right)\;, \; 2\leq i \leq N, \\
\end{array}\right.
\end{equation}
where $C_{1}=max\left\lbrace c, c \Vert a \Vert_{L^{\infty}}, c\left\vert\lambda\right\vert  \right\rbrace $, $C_{2}=max\left\lbrace c, c\left\vert\lambda\right\vert  \right\rbrace $ and $\theta=1-\frac{\alpha-1}{4}$.

\medskip

So, if $u_{i}\in X_{T,a}$, where $1\leq i \leq N$,  we have
\begin{equation}\label{A5}
\left\{\begin{array}{lll} \left\Vert\Phi(u_{1})(t)\right\Vert_{L^{r}\left(L^{\alpha+1}\right)}\leq C_{1} \left\Vert\varphi_{1}\right\Vert_{L^{2}}+ C_{1}\,T\,b +C_{1}\, T^{\theta}\,b^{\alpha}, \\
\left\Vert\Phi(u_{i})(t)\right\Vert_{L^{r}\left(L^{\alpha+1}\right)}\leq C_{2} \left\Vert\varphi_{i}\right\Vert_{L^{2}}+C_{2}\, T^{\theta}\,b^{\alpha}\;, \; 2\leq i \leq N. \\
\end{array}\right.
\end{equation}
Then, we have
\begin{equation}\label{A6}
\left\Vert\Phi(u_{i})(t)\right\Vert_{L^{r}\left(L^{\alpha+1}\right)}\leq C \left\Vert\varphi_{i}\right\Vert_{L^{2}}+C\, T\, b+C\, T^{\theta}\,b^{\alpha}\;, \; 1\leq i \leq N,
\end{equation}
where $C=max\lbrace C_{1}+C_{2}, C_{1}\rbrace$.

\medskip

Similarly, we have
\begin{equation}\label{A7}
\sup_{t\in[0,T]} \Vert \Phi(u_{i})(t)\Vert_{L^{2}}\leq C \Vert \varphi_{i}\Vert_{L^{2}}+C\, T\, b+C\, T^{\theta}\,b^{\alpha}\;, \; 1\leq i \leq N,
\end{equation}
where the positive constant $C$ depends on $\alpha$, $\lambda$ and the $L^{\infty}$ norm of the function $a$.
Then, from equation $(\ref{A6})$ and $(\ref{A7})$, we have
\begin{equation}\label{A8}
\vert \vert\vert \Phi(u_{i})(t) \vert\vert\vert_{T}\leq C \left\Vert\varphi_{i}\right\Vert_{L^{2}}+C\, T\, b+C\, T^{\theta}\,b^{\alpha}\;, \; 1\leq i \leq N.
\end{equation}
We fixe $b=2C\left\Vert\varphi_{i}\right\Vert_{L^{2}}\;, \; 1\leq i \leq N. $
Inequality $(\ref{A8})$ enable us to deduce
\begin{equation}
\vert \vert\vert \Phi(u)(t) \vert\vert\vert_{T}=\sum_{i=1}^{N} \vert \vert\vert \Phi(u_{i})(t) \vert\vert\vert_{T}
  \leq C \left\Vert\varphi\right\Vert_{L^{2}}\left( 1+ 2C\, T+2^{\alpha}C^{\alpha}\,T^{\theta}\left\Vert\varphi\right\Vert_{L^{2}}^{\alpha-1}\right).
\end{equation}
So, by choosing $T=\min(K_{1},K_{2})$, such that\begin{equation*}
K_{1}<\frac{1}{4C}\qquad \text{and}\qquad K_{2}<\frac{1}{\left(2^{\alpha+1}C^{\alpha}\Vert\varphi\Vert_{L^{2}}^{\alpha-1} \right)^{{1}/{\theta}}},
\end{equation*}
we get
\begin{equation}\label{A9}
2C\, T+2^{\alpha}C^{\alpha}\,T^{\theta}\left\Vert\varphi\right\Vert_{L^{2}}^{\alpha-1}<1.
\end{equation}
It follows that $\Phi$ is well-defined on $X_{T,b}$, that is $\Phi:X_{T,b}\longrightarrow X_{T,b}$.

\medskip

\item Step 2. Now, if $u,v\in X_{T,b}$, we have
\begin{equation*}
\left\{\begin{array}{lll}
 \left( \Phi(u_{1})-\Phi(v_{1})\right) (t)= \\
i  \int_{0}^{T} e^{i\partial_{x}^{2}(t-s)}\left[\lambda \left(\vert u_{1}\vert^{\alpha-1}u_{1}-\vert v_{1}\vert^{\alpha-1}v_{1}\right)+ia(.)(u_{1}-v_{1}) \right] (s)ds, \\
\left( \Phi(u_{i})-\Phi(v_{i})\right) (t)=i \lambda  \int_{0}^{T} e^{i\partial_{x}^{2}(t-s)}\left(\vert u_{i}\vert^{\alpha-1}u_{i}-\vert v_{i}\vert^{\alpha-1}v_{i}\right)(s)ds\;, \; 2\leq i \leq N. \\
\end{array}\right.
\end{equation*}
The same argument as in $(\ref{A3})$, show that
\begin{equation*}
\left\{\begin{array}{lll}
 \Vert \Phi(u_{1})-\Phi(v_{1})) (t)\Vert_{L^{r}(L^{\alpha+1})}
\leq
\\
c\int_{0}^{T} \left\Vert a(.)\left( u_{1}-v_{1}\right) (t)
\right \Vert_{L^{2}}dt
                  +c\left\vert\lambda\right\vert \left( \int_{0}^{T} \left\Vert \left\vert u_{1} \right\vert^{\alpha-1}u_{1}-\vert v_{1}\vert^{\alpha-1}v_{1} \right\Vert_{L^{\frac{\alpha+1}{\alpha}}}^{r'} dt\right)^{{1}/{r'}},  \\
\Vert( \Phi(u_{i})-\Phi(v_{i})) (t)\Vert_{L^{r}(L^{\alpha+1})} \leq
 \\
c\left\vert\lambda\right\vert \left( \int_{0}^{T} \left\Vert \left\vert u_{i} \right\vert^{\alpha-1}u_{i}-\vert v_{i}\vert^{\alpha-1}v_{i}\right \Vert_{L^{\frac{\alpha+1}{\alpha}}}^{r'} dt\right)^{{1}/{r'}}\;, \; 2\leq i \leq N. \\

\end{array}\right.
\end{equation*}
Since $
\left\vert \left\vert u_{i} \right\vert^{\alpha-1}u_{i}-\vert v_{i}\vert^{\alpha-1}v_{i}\right\vert\leq c \vert u_{i}-v_{i}\vert \left( \left\vert u_{i} \right\vert^{\alpha-1}+\vert v_{i}\vert^{\alpha-1}\right) \;, \; 1\leq i \leq N,$
then by H\"{o}lder's inequality, we obtain that
\begin{equation*}
\left\{\begin{array}{lll}
  \Vert(\Phi(u_{1})-\Phi(v_{1})) (t) \Vert_{L^{r}(L^{\alpha+1})}\leq c \Vert a\Vert_{L^{\infty}} T\underset{t\in[0,T]}{\sup}
\left\Vert u_{1}-v_{1}\right \Vert_{L^{2}} \\
+c\left\vert\lambda\right\vert T^{\theta} \left\lbrace
 \left( \int_{0}^{T}\Vert u_{1} \Vert_{L^{\alpha+1}}^{r}dt\right)^{{\alpha-1}/{r}}+\left( \int_{0}^{T}\Vert v_{1} \Vert_{L^{\alpha+1}}^{r}dt\right)^{{\alpha-1}/{r}}
 \right\rbrace  \times \\
\left(  \int_{0}^{T} \Vert u_{1}(t)-v_{1}(t)\Vert_{L^{\alpha+1}}^{r}dt\right)^{{1}/{r}},\\

\Vert( \Phi(u_{i})-\Phi(v_{i})) (t)\Vert_{L^{r}(L^{\alpha+1})}\leq
\\
 c\left\vert\lambda\right\vert T^{\theta} \left\lbrace
 \left( \int_{0}^{T}\Vert u_{i} \Vert_{L^{\alpha+1}}^{r}dt\right)^{{\alpha-1}/{r}}+\left( \int_{0}^{T}\Vert v_{i} \Vert_{L^{\alpha+1}}^{r}dt\right)^{{\alpha-1}/{r}}
 \right\rbrace \\
  \times \left(  \int_{0}^{T} \Vert u_{i}(t)-v_{i}(t)\Vert_{L^{\alpha+1}}^{r}dt\right)^{{1}/{r}}\;, \; 2\leq i \leq N.
\end{array}\right.
\end{equation*}
So, we have
\begin{equation}\label{A10}
\begin{array}{ll}
\Vert( \Phi(u)-\Phi(v)) (t)\Vert_{L^{r}(L^{\alpha+1})}=\sum_{i=1}^{N}\left \Vert\left( \Phi(u_{i})-\Phi(v_{i})\right) (t)\right \Vert_{L^{r}(L^{\alpha+1})}\\
\leq  C\,T\underset{t\in[0,T]}{\sup}
\left\Vert u_{1}-v_{1}\right \Vert_{L^{2}}
+2C\,T^{\theta} b^{\alpha-1} \left(  \int_{0}^{T} \Vert u(t)-v(t)\Vert_{L^{\alpha+1}}^{r}dt\right)^{{1}/{r}},
\end{array}
\end{equation}
where $C=\max\left\lbrace c \Vert a\Vert_{L^{\infty}}, c\left\vert\lambda\right\vert\right\rbrace $.
Similarly, we have
\begin{equation}
\label{A11}
\sup_{t\in[0,T]}\Vert( \Phi(u)-\Phi(v)) (t)\Vert_{L^{2}}
\leq
\end{equation}
$$
C\,T\sup_{t\in[0,T]}
\left\Vert u_{1}-v_{1}\right \Vert_{L^{2}}
+2C\,T^{\theta} b^{\alpha-1} \left(  \int_{0}^{T} \Vert u(t)-v(t)\Vert_{L^{\alpha+1}}^{r}dt\right)^{{1}/{r}}.
$$
Combining $(\ref{A10})$ and $(\ref{A11})$, we obtain
\begin{equation}
\vert \vert\vert ( \Phi(u)-\Phi(v)) (t) \vert\vert\vert_{T}\leq \left( C\,T+2\,C\,T^{\theta}\, b^{\alpha-1}\right)\vert \vert \vert u-v \vert\vert\vert_{T}.
\end{equation}
It follows from the choice of $b$, $b\leq 2\,C\Vert\varphi\Vert_{L^{2}}$ and inequality $(\ref{A9})$\begin{equation}
 C\,T+2\,C\,T^{\theta}\, b^{\alpha-1}\leq   2\,C\,T+2^{\alpha}\,C^{\alpha}\,T^{\theta}\, \Vert\varphi\Vert_{L^{2}}^{\alpha-1}<1.
\end{equation}

So, $\Phi$ is a contraction from $X_{T,b}$ into itself, then we have proved the existence and uniqueness of the solution of the problem
\begin{equation}\label{A12}
\left\{\begin{array}{lll}
 u_{1}(t)&=e^{it\partial_{x}^{2}}\varphi_{1}+i\int_{0}^{T} e^{i\partial_{x}^{2}(t-s)}\left( \lambda \vert u_{1}\vert^{\alpha-1}u_{1}+i\,a(.)u_{1}\right) (s)ds, \\
u_{i}(t)&=e^{it\partial_{x}^{2}}\varphi_{i}+i\int_{0}^{T} e^{i\partial_{x}^{2}(t-s)}\left( \lambda \vert u_{i}\vert^{\alpha-1}u_{i}\right) (s)ds\;, \; 2\leq i \leq N. \\
\end{array}\right.
\end{equation}
\item Step 3. Note that if $u, v$ are the corresponding solutions of (\ref{A12}) with initial data $\varphi, \psi$, respectively, then
\begin{equation*}
\left\{\begin{array}{lll}
 u_{1}-v_{1}= \\
e^{it\partial_{x}^{2}}\left(\varphi_{1}-\psi_{1} \right) +i\int_{0}^{T} e^{i\partial_{x}^{2}(t-s)}\left[\lambda \left(\vert u_{1}\vert^{\alpha-1}u_{1}-\vert v_{1}\vert^{\alpha-1}v_{1}\right)+ia(.)(u_{1}-v_{1}) \right] (s)ds, \\
u_{i}-v_{i}=e^{it\partial_{x}^{2}}\left(\varphi_{i}-\psi_{i} \right)+i \lambda  \int_{0}^{T} e^{i\partial_{x}^{2}(t-s)}\left(\vert u_{i}\vert^{\alpha-1}u_{i}-\vert v_{i}\vert^{\alpha-1}v_{i}\right)(s)ds\;, \; 2\leq i \leq N. \\
\end{array}\right.
\end{equation*}
Similarly, following the same arguments used earlier, we get
\begin{equation} \label{A13}
\left\{\begin{array}{lll}

  \Vert(u_{1}-v_{1}) (t) \Vert_{L^{r}(L^{\alpha+1})}\leq C\Vert \varphi_{1}-\psi_{1}\Vert_{L^{2}} + C\,T \underset{t\in[0,T]}{\sup}
\left\Vert u_{1}-v_{1}\right \Vert_{L^{2}} \\+C^{\alpha}2^{\alpha-1} T^{\theta}
\left\lbrace  \left\Vert\varphi_{1} \right\Vert_{L^{2}}^{\alpha-1}+  \left\Vert\psi_{1} \right\Vert_{L^{2}}^{\alpha-1}\right\rbrace
 \times \left(  \int_{0}^{T} \Vert u_{1}(t)-v_{1}(t)\Vert_{L^{\alpha+1}}^{r}dt\right)^{{1}/{r}},\\

 \Vert( u_{i}-v_{i}(t)\Vert_{L^{r}(L^{\alpha+1})}\leq  C\Vert \varphi_{i}-\psi_{i}\Vert_{L^{2}} +C^{\alpha}2^{\alpha-1} T^{\theta}\left\lbrace  \left\Vert\varphi_{i} \right\Vert_{L^{2}}^{\alpha-1}+  \left\Vert\psi_{i} \right\Vert_{L^{2}}^{\alpha-1}\right\rbrace \\
  \times \left(  \int_{0}^{T} \Vert u_{i}(t)-v_{i}(t)\Vert_{L^{\alpha+1}}^{r}dt\right)^{{1}/{r}}\;, \; 2\leq i \leq N, \\
\end{array}\right.
\end{equation}
where $C=\max\left\lbrace c, c \Vert a\Vert_{L^{\infty}}, c\left\vert\lambda\right\vert\right\rbrace.$
Analogously, we have
\begin{equation} \label{A14}
\left\{\begin{array}{lll}
 \underset{t\in[0,T]}{\sup} \Vert(u_{1}-v_{1}) (t) \Vert_{L^{2}}\leq C\Vert \varphi_{1}-\psi_{1}\Vert_{L^{2}} + C\,T \underset{t\in[0,T]}{\sup}
\left\Vert u_{1}-v_{1}\right \Vert_{L^{2}} \\+C^{\alpha}2^{\alpha-1} T^{\theta}
\left\lbrace  \left\Vert\varphi_{1} \right\Vert_{L^{2}}^{\alpha-1}+  \left\Vert\psi_{1} \right\Vert_{L^{2}}^{\alpha-1}\right\rbrace
 \times \left(  \int_{0}^{T} \Vert u_{1}(t)-v_{1}(t)\Vert_{L^{\alpha+1}}^{r}dt\right)^{{1}/{r}},\\

 \underset{t\in[0,T]}{\sup}\Vert( u_{i}-v_{i}(t)\Vert_{L^{2}}\leq  C\Vert \varphi_{i}-\psi_{i}\Vert_{L^{2}}
+C^{\alpha}2^{\alpha-1} T^{\theta}\left\lbrace  \left\Vert\varphi_{i} \right\Vert_{L^{2}}^{\alpha-1}+  \left\Vert\psi_{i} \right\Vert_{L^{2}}^{\alpha-1}\right\rbrace \\
  \times \left(  \int_{0}^{T} \Vert u_{i}(t)-v_{i}(t)\Vert_{L^{\alpha+1}}^{r}dt\right)^{{1}/{r}}\;, \; 2\leq i \leq N,\\
\end{array}\right.
\end{equation}
where $C=\max\left\lbrace c, c \Vert a\Vert_{L^{\infty}}, c\left\vert\lambda\right\vert\right\rbrace.$ Combining $(\ref{A13})$ and $(\ref{A14})$
\begin{equation}
\vert \vert\vert u-v \vert\vert\vert_{T}=\sum_{i=1}^{N} \vert \vert\vert u_{i}-v_{i} \vert\vert\vert_{T}
  \leq
	\end{equation}
$$
C \left\Vert \varphi - \psi\right\Vert_{L^{2}}+\left( C\, T+2^{\alpha-1}C^{\alpha}\,T^{\theta}\left\lbrace \left\Vert\psi\right\Vert_{L^{2}}^{\alpha-1}+\left\Vert\varphi\right\Vert_{L^{2}}^{\alpha-1} \right\rbrace\right)\vert \vert\vert v-u \vert\vert\vert_{T}.
$$
If $\left\Vert \varphi - \psi\right\Vert_{L^{2}}$ is small enough (see $  (\ref{A9})$), then\begin{equation*}
\vert \vert\vert u-v \vert\vert\vert_{T}\leq \tilde{k}\left\Vert \varphi - \psi\right\Vert_{L^{2}},
\end{equation*}
Consequently, we have proved the continuous dependence of $\Phi(u(t))=\Phi_{\varphi}(u(t))$ with respect to $\varphi$. This completes the proof of case I).
\medskip
\item[II) \textbf{Critical case,} $\alpha=5.$]
In this case, we need some modifications in the proof of case I) given above.\\
According to \cite{Cazenav,linares}, we have for $(p,q)$ be a pair satisfying condition $(\ref{couple})$ in Theorem $\ref{strichartz}$:
Given $\varphi \in L^{2}({\mathcal{R}})$ and $\varepsilon>0$, there is $\delta$ and $T>0$ such that if \begin{equation}
\Vert\varphi-\psi\Vert_{L^{2}}<\delta,
\end{equation}
then,\begin{equation}\label{c1}
\left(\int_{0}^{T}\left\Vert e^{it\partial_{x}^{2}} \psi\right\Vert_{L^{p}}^{q} dt\right)^{{1}/{q}} <\varepsilon.
\end{equation}
Let us define the complete metric space
\begin{equation*}
\tilde{X}_{T,b}=\left\lbrace
v\in C\left( \left[0,T \right];L^{2}(\mathcal{R} \right) \cap  L^6((0,T);L^{6}(\mathcal{R})),\vert \vert\vert v \vert\vert\vert\leq b
\right\rbrace,
\end{equation*}
where $\vert \vert\vert . \vert\vert\vert$ given by
$$\vert \vert\vert v \vert\vert\vert_{T}=\sup_{t\in[0,T]} \Vert v(t)-e^{it\partial_{x}^{2}}\Vert_{L^{2}} + \left( \int_{0}^{T}\|v(t)\|_{L^{6}}^{6}dt\right)^{{1}/{6}}.
$$
\item Step 1. Applying Theorem $\ref{strichartz}$ and Corollary $\ref{str}$ to system $(\ref{A1})$, it follows that
\begin{equation}\label{c2}
\left\{\begin{array}{lll}
  \underset{t\in[0,T]}{\sup}\big\Vert\Phi(u_{1})(t)-e^{it\partial_{x}^{2}}\varphi_{1}\big\Vert_{L^{2}}\leq  c \Vert a \Vert_{L^{\infty}}T\underset{t\in[0,T]}{\sup}\Vert u_{1}-e^{it\partial_{x}^{2}}\varphi_{1}(t)\Vert_{L^{2}}\\+c \Vert a \Vert_{L^{\infty}}T\underset{t\in[0,T]}{\sup}\Vert e^{it\partial_{x}^{2}}\varphi_{1}(t)\Vert_{L^{2}}+c\left\vert \lambda \right\vert \left(\int_{0}^{T}\Vert u_{1}(t)\Vert_{L^{6}}^{6}dt \right)^{{5}/{6}}, \\
 \underset{t\in[0,T]}{\sup}\Vert\Phi(u_{i})(t)-e^{it\partial_{x}^{2}}\varphi_{i}\Vert_{L^{2}}\leq  c\left\vert \lambda \right\vert \left(\int_{0}^{T}\Vert u_{i}(t)\Vert_{L^{6}}^{6}dt \right)^{{5}/{6}} \;, \; 2\leq i \leq N. \\
\end{array}\right.
\end{equation}
Therefore, we have
\begin{equation} \label{c4}
\sup_{t\in[0,T]}\left\Vert\Phi(u_{i})(t)-e^{it\partial_{x}^{2}}\varphi_{i}\right\Vert_{L^{2}}\leq C\,T\,b+C\,T\varepsilon+C\,b^{5}  \;, \; 1\leq i \leq N,
\end{equation}
where $C=\max\left\lbrace  c \Vert a \Vert_{L^{\infty}},c\left\vert \lambda \right\vert \right\rbrace .$ Similarly, by applying Theorem $\ref{strichartz}$, Corollary $\ref{str}$ and the estimate $(\ref{c1})$ to the system $(\ref{A1})$ , we obtain
\begin{equation}\label{c3}
\left\{\begin{array}{lll}
\left( \int_{0}^{T}\left\Vert\Phi(u_{1})(t)\right\Vert_{L^{6}}^{6}dt\right)^{{1}/{6}} \leq c\;\varepsilon+c \Vert a \Vert_{L^{\infty}}T\underset{t\in[0,T]}{\sup}\Vert u_{1}-e^{it\partial_{x}^{2}}\varphi_{1}(t)\Vert_{L^{2}}\\ +c \Vert a \Vert_{L^{\infty}}T\underset{t\in[0,T]}{\sup}\Vert e^{it\partial_{x}^{2}}\varphi_{1}(t)\Vert_{L^{2}}+c\left\vert \lambda \right\vert \left(\int_{0}^{T}\Vert u_{1}(t)\Vert_{L^{6}}^{6}dt \right)^{{5}/{6}},  \\
\left( \int_{0}^{T}\left\Vert\Phi(u_{i})(t)\right\Vert_{L^{6}}^{6}dt\right)^{{1}/{6}} \leq c\;\varepsilon+ c\left\vert \lambda \right\vert \left(\int_{0}^{T}\Vert u_{i}(t)\Vert_{L^{6}}^{6}dt \right)^{{5}/{6}}  \;, \; 2\leq i \leq N. \\
\end{array}\right.
\end{equation}
Then, we have
\begin{equation}\label{c5}
\left( \int_{0}^{T}\left\Vert\Phi(u_{i})(t)\right\Vert_{L^{6}}^{6}dt\right)^{{1}/{6}}\leq C\,\varepsilon+ C\,T\,b+C\,T\varepsilon+C\,b^{5}  \;, \; 1\leq i \leq N.
\end{equation}
We get from $(\ref{c4})$ and $(\ref{c5})$ that
\begin{equation}
\vert \vert\vert \Phi(u)(t) \vert\vert\vert_{T}\leq C\,\varepsilon+ C\,T\,b+C\,T\varepsilon+C\,b^{5}.
\end{equation}
Therefore, if
\begin{equation}\label{B}
C\,\varepsilon+ C\,T\,b+C\,T\varepsilon+C\,b^{5}< b,
\end{equation}
we get that $\Phi(\tilde{X}_{T,b})\subseteq \tilde{X}_{T,b}$.
\item Step 2. The argument used in the proof of Theorem \ref{Th1} yields
\begin{equation*}
\left\{\begin{array}{lll}
  \underset{t\in[0,T]}{\sup}\Vert(\Phi(u_{1})-\Phi(v_{1})) (t) \Vert_{L^{2}}\leq
	c \Vert a\Vert_{L^{\infty}} T\underset{t\in[0,T]}{\sup}
\left\Vert (u_{1}-e^{it\partial_{x}^{2}})-(v_{1}-e^{it\partial_{x}^{2}})\right \Vert_{L^{2}} \\
+2c\left\vert\lambda\right\vert b^{4} \left(  \int_{0}^{T} \Vert u_{1}(t)-v_{1}(t)\Vert_{L^{6}}^{6}dt\right)^{{1}/{6}},\\

 \underset{t\in[0,T]}{\sup}\Vert( \Phi(u_{i})-\Phi(v_{i})) (t)\Vert_{L^{2}}\leq
 2c\left\vert\lambda\right\vert b^{4}
  \left(  \int_{0}^{T} \Vert u_{i}(t)-v_{i}(t)\Vert_{L^{6}}^{6}dt\right)^{{1}/{6}}\;, \; 2\leq i \leq N. \\
\end{array}\right.
\end{equation*}
Similarly, we have
\begin{equation*}
\left\{\begin{array}{lll}
 \bigg(\int_{0}^{T}   \Vert( \Phi(u_{1})-\Phi(u_{1}))(t) \Vert_{L^{6}}^{6}dt\bigg)^{{1}/{6}} \leq \\
c \Vert a\Vert_{L^{\infty}} T\underset{t\in[0,T]}{\sup}
\left\Vert (u_{1}-e^{it\partial_{x}^{2}})-(v_{1}-e^{it\partial_{x}^{2}})\right \Vert_{L^{2}} +\\
2c\left\vert\lambda\right\vert b^{4} \left(  \int_{0}^{T} \Vert u_{1}(t)-v_{1}(t)\Vert_{L^{6}}^{6}dt\right)^{{1}/{6}},\\
 \bigg(\int_{0}^{T}   \Vert( \Phi(u_{i})-\Phi(v_{i}))(t) \Vert_{L^{6}}^{6}dt\bigg)^{{1}/{6}} \leq \\
 2c\left\vert\lambda\right\vert b^{4}
  \left(  \int_{0}^{T} \Vert u_{i}(t)-v_{i}(t)\Vert_{L^{6}}^{6}dt\right)^{{1}/{6}}\;, \; 2\leq i \leq N. \\
\end{array}\right.
\end{equation*}
This yields,\begin{equation*}
\vert \vert\vert (\Phi(u)-\Phi(v))(t) \vert\vert\vert_{T}\leq \left(C\,T+2C\,b^{4} \right) \vert \vert\vert u-v \vert\vert\vert_{T},
\end{equation*}
thus, for \begin{equation}\label{C}
C\,T+2C\,b^{4}<1,
\end{equation}
we have that $\Phi(.)$ is a contraction. Now, taking $b=2C\left\Vert \varphi \right\Vert_{L^{2}}$ and $\varepsilon= \left\Vert \varphi \right\Vert_{L^{2}}$, such that
\begin{equation*}
T\leq \dfrac{1-2^{5}C^{5}\varepsilon^{4}}{1+2C},
\end{equation*}
we see that both $(\ref{B})$ and $(\ref{C})$ are verified. This completes the proof, the remainder of the proof follows the same argument employed to show case I).
\end{proof}
We have the following result:
\begin{corollary}\label{coro2}
The solution $u$ of the system $(\ref{sys})$ obtained in Theorem $\ref{Th1}$ belongs to $L^{q}((0,s);L^{p}(\mathcal{R}))$ for all $(p,q)$ admissible pair defined in Theorem $\ref{strichartz}$.
\end{corollary}
\begin{proof}
The proof of this results is similar to the one given in \cite[Corollary 5.1]{linares} for the subcritical case and
\cite[Theorem 4.7.1]{Cazenav} for the critical case.
\end{proof}
\begin{remark}
Notice that the time of existence in the subcritical case depends only on $\left\Vert \varphi \right\Vert_{L^{2}}$; meanwhile, in the critical case, the time of existence depends on the $\varphi$ itself, and not only on its norm.
\end{remark}
The following corollaries establish global solution of the system $(\ref{sys})$ in $L^{2}$-norm in subcritical case and critical case respectively.
\begin{corollary}
If the nonlinearity power $\alpha\in(1,5)$, then for any $\varphi\in L^{2}(\mathcal{R})$ the local solution $u$ of the system $(\ref{sys})$ extends globally with
$$ u\in \mathcal{C}([0,+\infty);L^2(\mathcal{R}))\cap L_{loc}^{q}\big([0,+\infty);L^{p}(\mathcal{R})\big),$$
where $(p,q)$ satisfies the condition $(\ref{couple}).$
\end{corollary}
\begin{proof}
Since $T$ depends only on $\Vert \varphi \Vert_{L^{2}(\mathcal{R})}$ and, by using $(\ref{E})$,  we have $\int_{0}^{+\infty}\vert u \vert^{2}\leq \Vert\varphi\Vert_{L^{2}(\mathcal{R})}^{2}$, we deduce, after an interaction argument, that a similar inequality as in $(\ref{A9})$ remains valid for all $T>0$. This last fact enable us to conclude that $u(t)$ can be extended to all $[0,+\infty).$
\end{proof}
\medskip
The situation for the critical case $\alpha=5$ is quite different. In this case, the local result shows the existence of a solution in a time interval depending on the data $\varphi$ itself and not its norm. So, the fact that $\int_{0}^{+\infty}\vert u \vert^{2}\leq \Vert\varphi\Vert_{L^{2}(\mathcal{R})}^{2}$ does not guarantee the existence of a global solution. An important result of global solutions for this case is established provided that is assumed a smallness condition on the initial data.
According to \cite{Cazenav,linares}, we have:
\begin{corollary}\label{coro3}
Let us assume $\alpha=5$. The additional assumption $\Vert\varphi\Vert_{L^{2}}\ll 1$ implies that the local solution $u$ of the system $(\ref{sys})$ can be extended globally, that is\begin{equation}
u \in \mathcal{C}([0,+\infty);L^2(\mathcal{R}))\cap L^q\big([0,+\infty);L^{p}(\mathcal{R})\big), \forall\; T>0,
\end{equation}
where $(p,q)$ are admissible pairs satisfying condition $(\ref{couple})$.
\end{corollary}

\section{Exponential stability}

First, we give the following technical lemma:
\begin{lemma} \label{lem1}
The solution related to the system (\ref{sys}) verifies the following inequality:
\begin{equation}
\label{ener}
\int_0^t E_u (s) \, ds \leq \int_0^t \sum_{i=2}^N E_{u_i} (s) \, ds \, +
\end{equation}
$$
\int_0^t \int_0^R |u_1(s,x)|^2 \, dxds
- \frac{1}{2\alpha_0} \big(E_u (t)- E_u (0)\big), \, \forall\, t\geq 0.
$$
\end{lemma}

\begin{proof}
It is clear that:
\begin{equation}\label{prlem1}
\begin{array}{lll}
\int_0^tE_{u_1}(s)ds &=& \int_0^t\int_0^R|u_1(s,x)|^2dxds+\int_0^t\int_R^{+\infty}|u_1(s,x)|^2dxds\\
&\leq & \int_0^t\int_0^R|u_1(s,x)|^2dxds+\int_0^t\int_R^{+\infty}\frac{a(x)}{\alpha_0}|u_1(s,x)|^2dxds\\
&\leq & \int_0^t\int_0^R|u_1(s,x)|^2dxds+\int_0^t\int_0^{+\infty}\frac{a(x)}{\alpha_0}|u_1(s,x)|^2dxds\\
&= & \int_0^t\int_0^R|u_1(s,x)|^2dxds-\frac{1}{2\alpha_0}\big(E_u(t)-E_u(0)\big).\\
\end{array}
\end{equation}
Now, using the fact that $\dint_0^t E_u(s) ds= \dint_0^t \mathlarger{‎‎\sum}_{i=1}^N E_{u_i}(s)ds$, and (\ref{prlem1}), we obtain the desired result.
\end{proof}
Next, the following lemma is aimed to prove an estimate-type observability estimate.
\begin{lemma} \label{lem2}
Consider $\alpha=\lbrace 3,5 \rbrace$. Let $u$ be a solution associated to the system $(\ref{sys})$ with initial data $\varphi=(\varphi_1, \varphi_2,...,\varphi_N)\in L^2(\mathcal{R})$ satisfying $\|\varphi\|_{L^2(\mathcal{R})}\ll 1$ for the case $\alpha=5$. Then, for all $T\gg 1,$ there exists a positive constant $c$ which depends on $T$ such that the following inequality holds,
\begin{equation}\label{obs}
\int_0^T\int_0^{+\infty}a(x)|u_1(s,x)|^2dxds\geq c \, \left(\dint_0^T\sum_{i=2}^N E_{u_i}(s)ds+ \int_0^T\int_0^R|u_1(s,x)|^2dxds  \right).
\end{equation}
\end{lemma}
\begin{proof}
We argue par contradiction. We suppose that (\ref{obs}) is not true and let $\{(u^k_i)(0)\}_{k\in \mathbb{N}}, 1\leq i\leq N$ be a sequence of initial data attached with the solutions $\{u^k\}_{k\in \mathbb{N}}=\{(u^k_1, u^k_2,..., u^k_N)\}_{k\in \mathbb{N}},$ which is assumed to be uniformly bounded by a constant $C>0,$ verify
\begin{equation}\label{cont}
\dint_0^T\int_0^{+\infty}a(x)|u_1^k(s,x)|^2dxds\leq \frac{1}{k}\left(\dint_0^T\sum_{i=2}^N E_{u_i^k} (s)ds+ \int_0^T\int_0^R|u_1^k(s,x)|^2dxds  \right).
\end{equation}
On account of $ E_{u^k}(t) \leq E_{u^k}(0) \leq C,$ we obtain a subsequence of $\{u^k\}_{k\in  \mathbb{N}}$, still denoted by $\{u^k\}_{k\in  \mathbb{N}}$, which verifies the convergence
\begin{equation}
u^k \rightharpoonup u \;\; \mbox{ weakly in  } \;\; L^2((0,T); L^2(\mathcal{R})).
\end{equation}
Then, we deduce
\begin{equation}\label{3.6}
\lim_{k\rightarrow + \infty}\int_0^T\int_0^{+\infty}a(x)|u_1^k(s,x)|^2dxds=0.
\end{equation}
Consequently, we have
\begin{equation}\label{1}
\lim_{k\rightarrow +\infty}\int_0^T\int_R^{+\infty}|u_1^k(s,x)|^2dxds=0.
\end{equation}
On the other hand, by using Lemma \ref{lem1}, we deduce the existence of $T_1=T_1(C)>0$ ($T_1=T$ for the case $\alpha=5$) such that
\begin{equation}
\{u_1^k\}_{k\in \mathbb{N}} \;\; \mbox{ is bounded in } \;\; L^2((0,T_1); H^{1/2}(]0,R[)).
\end{equation}
Now, we consider the equation
\begin{equation}\label{eq}
i \partial_t u^k_{1}= - \partial^2_x u^k_{2} - \lambda |u^k_1|^{\alpha-1}u^k_1-ia(x)u^k_1, \;\;\; \mbox{in} \;\; \mathcal{D}^\prime ( (0,T_1)\times\mathbb{R}_+).
\end{equation}

First, we consider $\alpha=3$. Note that:

\medskip

\begin{itemize}
\item
The term $a(x)u^k_1$ is bounded in $L^6((0,T_1)\times \mathbb{R}_+)$ (Stricharz estimates) and $L^6((0,T_1)\times \mathbb{R}_+)\hookrightarrow L^2((0,T_1);H^{-2}_{loc}( \mathbb{R}_+)).$
\item
Similarly, the term $|u^k_1|^{2}u^k_1 \in L^2((0,T_1)\times \mathbb{R}_+)\hookrightarrow L^2((0,T_1);H^{-2}_{loc}( \mathbb{R}_+)).$
\item
Since $u^k_1$ is bounded in $L^2((0,T_1)\times \mathbb{R}_+)$ the term $\partial^2_x u^k_{2}$ is bounded in $L^2((0,T_1);H^{-2}_{loc}( \mathbb{R}_+)).$
\end{itemize}
The case $\alpha=5$ is analogous. The main difficulty is to deal with the nonlinear term $ |u^k_1|^{4}u^k_1$. By applying Theorem $\ref{Th1}$, Corollary $\ref{coro3}$ and Strichartz estimates, we have that  $ |u^k_1|^{4}u^k_1$ is bounded in $L^{{6}/{5}}((0,T)\times\mathbb{R}_{+})\hookrightarrow L^{{6}/{5}}((0,T);H_{loc}^{-2}\mathbb{R}_{+})$. the remainder of the conclusion is similar.

Thus, we deduce that $\partial_t u^k_{1}$ is bounded in $L^2((0,T_1);H^{-2}_{loc}( \mathbb{R}_+)),$ and we conclude, by using Lemma $\ref{lem4}$, the existence of a subsequence, still denoted by $\{u^k_1\}_{k\in \mathbb{N}}$, such that
\begin{equation}\label{2}
u^k_1\longrightarrow u_1 \;\; \mbox{ strong in} \;\; L^2((0,T_1)\times(0,R)).
\end{equation}
Besides we have $u^k_1 \rightarrow u_1$ a.e. in $ [0,T_1]\times [0,R].$ Using (\ref{1}) and (\ref{2}) we get
\begin{equation}
u^k_1 \rightarrow \tilde{u}_1 = \left\{\begin{array}{lll}
u_1, & & \mbox{ a.e. in } \;\; [0,T_1]\times [O,R], \\
0, & & \mbox{ a.e. in } \;\; [0,T_1]\times [R,+\infty[.\\
\end{array}\right.
\end{equation}
At this point, we will divide the proof into two cases.
\medskip

\item Case 1: $u\neq 0.$

\medskip
First, we consider the case $\alpha=3$, handling by the Strichartz inequalities we have $u^k_1$ is bounded in
$L^8((0,T_1),L^4(\mathbb{R}_+))\hookrightarrow L^4((0,T_1)\times \mathbb{R}_+),$ for all $k\in \mathbb{N}$.\\Therefore, $\{|u_1^k|^2u_1^k\}_{k\in \mathbb{N}}$ is bounded in $L^{4/3}((0,T_1),L^4(\mathbb{R}_+))$. Consequently, we obtain $|u_1^k|^2u_1^k \rightharpoonup |\tilde{u}_1|^2\tilde{u}_1$ in $L^{4/3}((0,T_1),L^4(\mathbb{R}_+))$.
Now, we can pass to the limit in (\ref{eq}), we find
\begin{equation}\label{limit}
 \left\{\begin{array}{lll}
i \partial_t u_{1} + \partial^2_x u_{1} + \lambda |u_1|^2u_1 = 0 & & \mbox{ a.e. in } \;\; [0,T_1]\times [0,R], \\
u_1=0, & & \mbox{ a.e. in } \;\; [0,T_1]\times [R,+\infty[.\\
\end{array}\right.
\end{equation}
Furthermore, since $u_1\in L^2((0,T_1);L^2(\mathbb{R}_+)),$ there is $t_0\in (0,T_1)$ such that $u_1(t_0,.)\in L^2(\mathbb{R}_+)$  and therefore $u_1\in \mathcal{C}((t_0,T_1); L^2(\mathbb{R}_+))\cap L^2((t_0,T_1); L^{\infty}(\mathbb{R}_+)).$ Incoming, we see that $u_1$ is a mild solution with initial data $\varphi_1$ having compact support and we use Theorem \ref{Th3} (see appendix) to conclude that $u_1\in \mathcal{C}^{\infty}((0,T_1)\times \mathbb{R}_+)$ with $u_1(t,x)=0,$ for all $(t,x)\in (0,T_1)\times (R,+\infty).$ Consequently, we obtain $u_1(t_1,.),$ $u_1(t_2,.)\in H^1(e^{\beta|x|^{\rho}}dx)$, for all $\beta>0$ and $\rho >2.$ Which gives $u_1\equiv 0$ in $ (0,T_1)\times [0,R]$ (see Theorem \ref{Th4} in appendix). We obtain finally $u_1\equiv 0$ in $ (0,T_1)\times\mathbb{R}_+.$

\medskip

In the sequel, we have

\begin{equation}\label{var} \begin{array}{lll}
\|\varphi_1^k\|_{L^2(\mathbb{R}_+)}^2=E_{u_{1}^k}(0) &=&E_{u_{1}^k}(T_1) + 2\dint_0^{T_1}\dint_0^{+\infty}a(x)|u_1^k(s,x)|^2dsdx\\
&\leq & \frac{1}{T_1}\dint_0^{T_1}E_{u_{1}^k}(s)ds + 2\dint_0^{T_1}\dint_0^{+\infty}a(x)|u_1^k(s,x)|^2 dsdx.\\
\end{array}
\end{equation}
Using estimations $(\ref{prlem1})$, we obtain
\begin{equation}
\|\varphi_1^k\|_{L^2(\mathbb{R}_+)}^2=E_{u_{1}^k}(0) \leq\left( 2+\frac{1}{T_{1}\alpha_{0}}\right) \dint_0^{T_1}\dint_0^{+\infty}a(x)|u_1^k(s,x)|^2 dxds+
\end{equation}
$$
\frac{1}{T_{1}}\dint_0^{T_1}\dint_0^{R}|u_1^k(s,x)|^2 dxds.
$$
Using $(\ref{3.6})$ and the fact that \begin{equation*}
u_{1}^{k}\longrightarrow 0 \qquad \text{in}\; L^{2}\left((0,T_{1})\times\mathbb{R} \right).
\end{equation*}
We get
\begin{equation}\label{eqE}
E_{u_{1}^k}(0)\underset{k\to+\infty}{\longrightarrow}\ 0.
\end{equation}

Therefore, since $\|u_1^k(t)\|^2_{L^2(\mathbb{R}_+)}\leq E_{u_{1}^k}(0)$, for all $t\in [0,T],$ we deduce that $u_1\equiv 0$ in $[0,T]\times \mathbb{R}_+.$
This gives a contradiction.\\
Now, we consider $\alpha=5$. The procedure is very similar. The main difference is to justify the passage to the limit at
system (\ref{limit}). Handling by the Strichartz inequalities and since $\Vert\varphi_{1}\Vert_{L^{2}(\mathbb{R_{+}})}\ll 1$, we have $u_{1}^{k}$ in bounded in $L^{6}((0,T)\times\mathbb{R}_{+})$, for all $k\in\mathbb{N}$. Therefore,  $ \lbrace|u^k_1|^{4}u^k_1\rbrace_{k\in\mathbb{N}}$ is bounded in $L^{{6}/{5}}((0,T)\times\mathbb{R}_{+})$. The remainder of the proof follows similarly as determinate in the case $\alpha=3$.
\medskip

\item Case 2: $u_1\equiv 0.$

\medskip

Let $\beta_k=\|u_1^k\|^2_{L^2((0,T)\times(0,R))}$.

\medskip

We have, by using Lebesgue's Dominated Convergence Theorem, that\begin{equation}\label{LDC}
\lim_{k\rightarrow+\infty}\beta_k=0.
\end{equation}
 Also, $v^k_1=\frac{u^k_1}{\beta_k}$ satisfies $\|v_1^k\|^2_{L^2((0,T)\times(0,R))}=1$ and verifies:
\begin{equation}\label{eqv}
i \partial_t v^k_{1} + \partial^2_x v^k_{2}+\lambda \beta_k^{2}|v^k_1|^{\alpha-1}v^k_1+ia(x)v^k_1 = 0, \;\; \mbox{in} \;\; \mathcal{D}^\prime ( (0,T_1)\times\mathbb{R}_+).
\end{equation}

By virtue of (\ref{cont}) and (\ref{LDC}) we get:
\begin{equation}\label{vr}
\lim_{k\rightarrow+\infty}\dint_0^{T}\int_0^{+\infty}a(x)|v_1^k(s,x)|^2dxds=0.
\end{equation}
Using the fact that $a(x)>\alpha_0>0$, we deduce from (\ref{vr}):
\begin{equation}\label{19}
\lim_{k\rightarrow+\infty}\dint_0^{T}\int_R^{+\infty}|v_1^k(s,x)|^2dxds=0,
\end{equation}
and this allow to find $v_1^k\rightarrow 0$ in $L^2\Big((0,T);L^2([R,+\infty[)\Big).$
Using the same arguments as in the case $u_1\neq 0$, (also in this case, we need to infer a bound for the nonlinear term  $ \lbrace|v^k_1|^{\alpha-1}v^k_1\rbrace_{k\in\mathbb{N}}$ in $L^m((0,T_2)\times \mathbb{R}_+)$, where $m={4}/{3}$ for $\alpha=3$ and $m={6}/{5}$ for $\alpha=5$), we conclude that $\partial_t v^k_{1}$, for some $T_2>0 \;(T_{2}=T, if \alpha=5)$, is bounded in $L^2((0,T_2);H^{-2}_{loc}( \mathbb{R}_+))$ and we get a function $\tilde{v}_1$ which verifies:

\begin{itemize}
\item
$v^k_1 \rightharpoonup \tilde{v}_1$ weakly in $L^2\Big((0,T_2),L^2(\mathbb{R}_+)\Big),$
\item
$\tilde{v}_1$ satisfies
\begin{equation}
\tilde{v}_1=\left\{\begin{array}{lll} v_1,  \;\;\; &\mbox{ a.e. in}& \;\; (0,T_2)\times [0,R],\\
0,\;\;  &\mbox{ a.e. in}&  (0,T_2)\times [R,+\infty[,
\end{array}\right.
\end{equation}
where $v_1$ is a solution of
\begin{equation}\label{v1}
\left\{\begin{array}{lll} i \partial_t v_{1} + \partial_x^2 v_{1} = 0,  \;\; &\mbox{ in}& \;\; \mathcal{D}^\prime \big((0,T_2)\times \mathbb{R}_+\big),\\
v_1=0,\;\;  &\mbox{ a.e. in}&  (0,T_2)\times [R,+\infty[.\\
\end{array}\right.
\end{equation}
\end{itemize}

Thanks to Holmogren's Theorem and (\ref{v1}) we obtain $v\equiv 0$ in $(0,T_2)\times [0,R].$ On the other hand, we use the fact that $v_1^k$ is bounded in $L^2\Big((0,T_2);H^{1/2}(0,R)\Big)$ and Aubin-Lions's Lemma (Lemma \ref{lem4} in appendix) to infer that:
\begin{equation}\label{convv}
v^k_1 \rightarrow 0  \;\; \mbox{ strong in}  \;\; L^2\Big((0,T_2)\times [0,R]\Big).
\end{equation}
Combining (\ref{19}) with (\ref{convv}) we get:
\begin{equation}
v^k_1 \rightarrow 0  \;\; \mbox{ strong in}  \;\; L^2\Big((0,T_2)\times [0,+\infty[\Big).
\end{equation}
Finally, we use the same arguments used to compute the limit in (\ref{eqE}) to obtain
\begin{equation}
v^k_1 \rightarrow 0  \;\; \mbox{ strong in}  \;\; L^2\Big((0,T)\times [0,+\infty[\Big).
\end{equation}
This is in contradiction with $\|v_1^k\|^2_{L^2((0,T)\times(0,R))}=1$.

\end{proof}
Now, we prove our main stability result (\ref{decay}).
\begin{proof}[Proof of Theorem \ref{theodec}]
Combining (\ref{ener}) and (\ref{obs}), we get:
\begin{equation}\begin{array}{lll}
\int^T_0E_u(s)ds&\leq & -\frac{1}{2c}\int^T_0\int^{+\infty}_0a(x)|u_1(s,x)|^2dxds -\frac{1}{2\alpha_0}\Big(E_u(T)-E_u(0)\Big) \\
&=& -\Big( \frac{1}{2c}+\frac{1}{2\alpha_0}\Big)\Big(E_u(T)-E_u(0)\Big).
\end{array}
\end{equation}
Using the fact that $t\mapsto E_u (t)$ is a nonincreasing function we deduce:
\begin{equation}
TE_u (T)\leq \int^T_0 E_u (s)ds \leq - C(T) \Big(E_u (T)-E_u (0)\Big),
\end{equation}
where $C(T):=\Big( \frac{1}{2c}+\frac{1}{2\alpha_0}\Big).$\\
This implies that
\begin{equation}
E_u(T) \leq \frac{C(T)}{C(T)+T}E_u(0).
\end{equation}
Finally, using the semigroup property, we obtain the exponential decay.
\end{proof}
\section*{Appendix}\label{Appendix}
In this appendix, we present some useful results used in the paper.
Let us first recall a result establishing a $C^{\infty}$ smoothing effect of nonlinear Sch\"{o}dinger equation.
\begin{theorem} (\cite[Theorem $5.7.3$]{Cazenav}) \label{Th3}
Consider $T>0$, $\lambda=\pm 1$ and $\alpha$ an odd positive number. Let $v$ be the global solution in $ C\left( \left[0,T \right];L^{2}(\mathcal{R}) \right) \cap  L^2\big((0,T);L^{\infty}(\mathcal{R})\big)$ of
\begin{equation}\label{IVP}
 \left\{\begin{array}{lll}
i \partial_t v + \partial^2_x v + \lambda |v|^{\alpha-1}v = 0,  &\qquad  (0,T)\times\mathbb{R},\\
v(x,0)=v_{0},   &\qquad x\in \mathbb{R}.
\end{array}\right.
\end{equation}
Then, $v\in C^{\infty}([0,T]\times\mathbb{R})$ for all $v_{0}\in L^{2}(\mathbb{R})$ with a compact support.
\end{theorem}

We have the same result for star-shaped and tadpole
graph geometries, see\cite{amm1} and \cite{amm1bis} respectively, for more details.

\medskip

We recall also the following unique continuation theorem for regular solutions of the nonlinear
equation in $(\ref{IVP})$. This result is more general in the sense that it deserves for nonlinear
Schr\"{o}dinger equation in the domain $(x, t)\in \mathbb{R}^{n}\times[0,T ], n\geq1$, with a general nonlinearity $F(v, \bar{v})$. In such case, we must consider $k \in\mathbb{Z}^{+}$ satisfying $k>\frac{n}{2}+1$.\\
Let us define the weighted Sobolev space $ H^{1}\left( e^{\beta\vert x\vert^{\rho}}dx\right)$, as\begin{equation*}
  H^{1}\left( e^{\beta\vert x\vert^{\rho}}dx\right)=\left\lbrace g;\int_{\mathbb{R}}\left\vert g(x)\right\vert^{2} e^{\beta\vert x\vert^{\rho}}dx+\int_{\mathbb{R}}\left\vert g'(x)\right\vert^{2} e^{\beta\vert x\vert^{\rho}}dx <\infty\right\rbrace .
\end{equation*}
\begin{theorem} (\cite[Theorem 2.1]{Escauriaza}) \label{Th4}
Let $w\in C\left([0,T];H^{k}(\mathbb{R}) \right) $, $k\in \mathbb{Z}^{+}$, $k>{3}/{2}$ be a strong solution of the equation in $(\ref{IVP})$  in the domain $(x, t)\in \mathbb{R}\times[0,T]$. If there is $t_{1},t_{2}\in[0,T], t_{1}\neq t_{2}$, $\rho>2$ and $\beta>0$ such that
\begin{equation*}
w(.,t_{1}), w(.,t_{2})\in H^{1}\left( e^{\beta\vert x\vert^{\rho}}dx\right),
\end{equation*}
then $w\equiv 0$.
\end{theorem}
Now, notice the following lemmas.
\begin{lemma}(Lions'Lemma \cite[Lemma 1.3]{AL}) \label{lem3}
Let $\omega$ be an open bounded subset of $\mathbb{R}\times \mathbb{R}$. Consider $\lbrace f_{n}\rbrace_{n\in\mathbb{N}}$ a sequence in $L^{q}(\omega)$, $1<q<\infty$, satisfying $\Vert f_{n}\Vert_{L^{q}(\omega)}\leq C$ and $f_{n}\longrightarrow f$ a.e. in $\omega$. Thus $f_{n}\rightharpoonup f$ in  $L^{q}(\omega)$, as $n\longrightarrow +\infty$.
\end{lemma}

\begin{lemma} (Aubin-Lions'Lemma \cite[Corollary 4]{SJ})\label{lem4}
Let $X_{0}, X$, and $X_{1}$ be three Banach spaces with $X_{0}\subset X\subset X_{1}$. Suppose that $X_{0}$ is compactly embedded in $X$ and that $X$ is continuously embedded in $X_{1}$. Suppose also that $X_{0}$ and $X_{1}$ are reflexive spaces. For $1<p,q<\infty$, let $W=\left\lbrace u\in L^{p}((0,T);X_{0}), \frac{du}{dt}\in L^{q}((0,T);X_{1})\right\rbrace $. Then the embedding of $W$ into $L^{p}((0,T);X)$ is compact.
\end{lemma}

\end{document}